\documentclass[12pt]{amsart}
\usepackage{latexsym}
\usepackage{euscript}
\usepackage{amsfonts}
\usepackage{amssymb}

\theoremstyle{plain}
\newtheorem{thm}{Theorem}
\newtheorem{Main}{Theorem\hskip -0.1ex}

\newtheorem{lem}[thm]{Lemma}

\newtheorem{prop}[thm]{Proposition}

\theoremstyle{definition}

\newtheorem{rem}[thm]{Remark}

\def\for{\text{for }}

\def\norm#1{\|#1\|}
\def\pref#1{(\ref{#1})}
\def\cx{C(X)}
\def\alg{\ifmmode {\mathcal A} \else ${\mathcal A}$\fi}
\def\lmu{L^{\infty}(\mu)}

\title {Fibers of $L^{\infty}$ algebra.}
\author{Marek Kosiek}
\address{Instytut Matematyki, Uniwersytet Jagiello\'nski, \L ojasiewicza 6,
30-348, Krak\'ow, Poland}
\email{Marek.Kosiek@im.uj.edu.pl}


\begin{document}

\begin{abstract}
It is shown that Gelfand transforms of elements $f\in\lmu$ are constant at almost every fiber $\Pi^{-1}(\{x\})$
of the spectrum of $\lmu$ in the following sense:
for each $f\in\lmu$ there is an open dense subset $U=U(f)$ of this spectrum having full measure and
such that the Gelfand transform of $f$ is constant on the intersection $\Pi^{-1}(\{x\})\cap U$.
\end{abstract}

\maketitle

The proof of the main
result bases on topological and measure properties of the
spectrum of $\lmu$,
(see \cite{D}, \cite{G} I.9).
This result is related to certain techniques connected with studying
abstract approach to A-measures problem and corona problem.

In this note we consider $\cx$, the algebra of all complex-valued continuous
functions on a compact space $X$. Moreover we assume
\begin{itemize}
\item[($*$)]
$\mu$ is a regular Borel probabilistic measure on
$X$ such that $X$ is equal to the closed support of $\mu$
\end{itemize}
The set $\lmu$ of equivalence classes $[f]$ of $[\mu]$ essentially bounded measurable functions $f$ on $X$
is a commutative C*-algebra under standard operations.

Let $Y$ be the spectrum of $\lmu$.
By Gelfand-Naimark theorem, $\lmu$ is isometrically isomorphic
(by the Gelfand transform $[f]\to\widehat{[f]}$) to $C(Y)$.

Let $y\in Y$ and define a functional $\Pi_y$ on $\cx$ as follows:
\begin{equation}\label{Pi}
\Pi_y(f):=\widehat{[f]}(y)\quad\for\ f\in\cx.
\end{equation}
Since
$$\Pi_y(fg)=\widehat{[fg]}(y)=\widehat{[f][g]}(y)=(\widehat{[f]}\widehat{[g]})(y)=\widehat{[f]}(y)\widehat{[g]}(y)
=\Pi_y(f)\Pi_y(g),$$
we conclude that $\Pi_y$ is a linear-multiplicative functional on $\cx$, so it can be identified with
some point in $X$. Using this identification we can write $f(\Pi_y)=\Pi_y(f)$ for $f\in\cx$. Consider the mapping $\Pi:y\to\Pi_y$ and
observe that $f\circ\Pi=\widehat{[f]}$ for $f\in\cx$. Hence $f\circ\Pi$ is a continuous function on $Y$
for each $f\in\cx$. Consequently,
since Gelfand topologies on $X$ and $Y$ are equal to the restrictions of
the weak-star topologies to $X$ and $Y$ respectively, we have the following

\begin{lem}\label{pi}
Projection $\Pi:Y\to X$ is continuous.
\end{lem}

Let us consider the sequence of mappings
\begin{equation}\label{ff}
\cx\ni f\to[f]\to\widehat{[f]}\in C(Y).
\end{equation}
By the assumption ($*$), the first mapping is an isometry into $\lmu$.
The last one is the Gelfand transform: $\lmu\to C(Y)$ which is also an isometry. Hence, by \pref{Pi} we have for $f\in\cx$
\begin{equation}\label{normy}
\sup_{x\in X} |f(x)|=\norm f=\norm{\widehat{[f]}}=\sup_{x\in\Pi(Y)} |f(x)|.
\end{equation}
By Lemma \ref{pi}, the set $\Pi(Y)$ is compact, and hence closed in $X$ which by
\pref{normy} implies that $\Pi(Y)$ contains Shilov boundary of $\cx$. Consequently $\Pi(Y)$ must be equal to $X$.
So we have
\begin{prop}\label{sub}\hfill
\begin{enumerate}
\item
Up to the isometric equivalence given by \pref{ff}, $\cx$ can be considered as a closed subalgebra of $\lmu$.
\item
Each element $x\in X$ as a linear-multiplicative functional on $\cx$ has a linear-multiplicative extension $y:[f]\to\widehat{[f]}(y)$
to the whole $\lmu$.
\end{enumerate}
\end{prop}

From now on we will not distinguish in writing Borel, $[\mu]$ essentially bounded functions on $X$ from their
equivalence classes in $\lmu$.
By the above consideration, if $f\in\cx$ then $\widehat f$ is constant on each fiber $\Pi^{-1}(\{x\})$ for $x\in X$.

Since we identify $\lmu$ with  $C(Y)$, Riesz Representation Theorem
gives a regular positive Borel measure
$\tilde\mu$ on $Y$ ''representing $\mu$'' i.e. such
 that $\norm{\tilde\mu}=\norm{\mu}$ and
\begin{equation}\label{tmu}
\int f\,d\mu=\int\hat f\,d\tilde\mu\quad \for\ f\in \lmu.
\end{equation}

For any Borel $E\subset X$ its characteristic function
$\widehat{\chi_E}$ as an idempotent in $C(Y)$ is of the form
$\chi_{U_E}$, thus assigning a closed-open set $U_E$ in $Y$ to any measurable
$E\subset X$.
Applying \pref{tmu} to $\chi_E$ we get for any Borel subset $E$ of X the equality
\begin{equation}\label{emu}
\mu(E)=\tilde\mu(U_E).
\end{equation}
Moreover (Lemma 9.1 and Corollary 9.2 of \cite{G}) we have

\begin{lem}\label{gam1}
The family $\{U_E: E\subset Y,\ E\ \text{measurable}\}$ form a basis for the topology of $Y$.
If $U$ is an open non-empty subset of $Y$, then $\tilde\mu(U)>0$.
\end{lem}

\begin{lem}\label{chi}
If $E,F$ are Borel subsets of $X$, and $E\subset F$
then $\widehat{\chi_E}\leqslant\widehat{\chi_F}$ and $U_E\subset U_F$.
\end{lem}

\begin{proof}
If $E\subset F$ then $\chi_E=\chi_E\cdot\chi_F$. Hence
$\widehat{\chi_E}=\widehat{\chi_E}\cdot\widehat{\chi_F}$
which means that $\widehat{\chi_E}\leqslant\widehat{\chi_F}$.
Since $\chi_{U_E}=\widehat{\chi_E}$ and
$\chi_{U_F}=\widehat{\chi_F}$, we have $U_E\subset U_F$.
\end{proof}

\begin{lem}\label{open}
If $E\subset X$ is open then $\Pi^{-1}(E)\subset U_E$ and
$\chi_{\Pi^{-1}(E)}\leqslant\widehat{\chi_E}$.
If $E\subset X$ is closed then $\Pi^{-1}(E)\supset U_E$ and
$\chi_{\Pi^{-1}(E)}\geqslant\widehat{\chi_E}$.
\end{lem}

\begin{proof}
Let $E$ be open in $X$ and $x\in E$. Then there is a continuous function
$f:X\to[0,1]$ such that $f(x)=1$ and $f\leqslant\chi_E$. Hence
$\hat{f}$ is equal 1 on $\Pi^{-1}(\{x\})$ and $f=f\cdot\chi_E$, which implies
$\hat f=\hat f\cdot\widehat{\chi_E}=\hat f\cdot\chi_{U_E}$.
Consequently $\widehat{\chi_{U_E}}$ is equal 1 on $\Pi^{-1}(\{x\})$ which means
that $\Pi^{-1}(\{x\})\subset U_E$. Since $x$ was an arbitrary point of $E$, we have
$\Pi^{-1}(E)\subset U_E$.
Then also $\chi_{\Pi^{-1}(E)}\leqslant\chi_{U_E}=\widehat{\chi_E}$.

If $E$ is closed then $X\setminus E$ is open and $\chi_E\cdot\chi_{X\setminus E}=0$,
$\chi_E+\chi_{X\setminus E}=1$. Consequently
$\chi_{U_E}\chi_{U_{X\setminus E}}=\widehat{\chi_E}\cdot\widehat{\chi_{X\setminus E}}=0$ and
$\chi_{U_E}+\chi_{U_{X\setminus E}}=\widehat{\chi_E}+\widehat{\chi_{X\setminus E}}=1$.
It means that $U_E\cap U_{X\setminus E}=\emptyset$ and $U_E\cup U_{X\setminus E}=Y$
which implies the desired statement for closed sets.
\end{proof}

\begin{rem}
Till now the regularity of $\mu$ has not been used.
\end{rem}

\begin{lem}\label{ue}
If $E$ is a Borel subset of $X$ then
\begin{equation}\label{eue}
\mu(E)=\tilde{\mu}({\Pi^{-1}(E)})=\tilde{\mu}(U_E).
\end{equation}
If $E\subset X$ is open then $\overline{\Pi^{-1}(E)}=U_E$.
If $E\subset X$ is closed then $\text{int}(\Pi^{-1}(E))=U_E$.
\end{lem}

\begin{proof}
By the regularity of $\mu$, for any $\varepsilon>0$ we can find a compact set $K\subset X$ and an open set $V\subset X$ such that
$K\subset E\subset V$ and $\mu(V\setminus K)<\varepsilon$. Also there exists $f\in\cx$ such that $\chi_K\leqslant f\leqslant\chi_V$.
By the continuity of $f$ we have $\chi_{\Pi^{-1}(K)}\leqslant \hat f\leqslant\chi_{\Pi^{-1}(V)}$.
(Proposition \ref{sub} and the consideration following it).
Hence $|\mu(E)-\int f\,d\mu|<\varepsilon$
and $|\tilde\mu(\Pi^{-1}(E))-\int \hat f\,d\tilde\mu|<\varepsilon$ which by \pref{tmu} and free choice of $\varepsilon$ gives
$\mu(E)=\tilde\mu(\Pi^{-1}(E))$. The second equality in \pref{eue} we get by \pref{emu}.

If $E$ is closed then $U_E\subset\Pi^{-1}(E)$ by Lemma \ref{open}. So $U_E\subset\text{int}(\Pi^{-1}(E))$ and
$\text{int}(\Pi^{-1}(E))\setminus U_E$ is open since $U_E$ is closed-open.
Consequently $\text{int}(\Pi^{-1}(E))=U_E$ by Lemma \ref{gam1}.

The assertion for open sets follows from the equalities $\text{int}(\Pi^{-1}(E))=Y\setminus\overline{\Pi^{-1}(X\setminus E)}$ and
$U_E=Y\setminus U_{X\setminus E}$.
\end{proof}

\begin{Main}
If $\mu$ is a measure satisfying ($*$), $Y$ is the spectrum of $\lmu$, and $h\in\lmu$,  then there exists an open dense subset
$U$ of $Y$ with $\tilde{\mu}(U)=\tilde{\mu}(Y)$ such that
$\hat h$ is constant on $\Pi^{-1}(\{x\})\cap U$ for all $x\in X$.
\end{Main}

\begin{proof}
Let $h\in\lmu$, and let $\varepsilon>0$. By Lusin Theorem there is $g\in C(X)$ with
$\norm g\leqslant\norm h$ and a closed set
$Z\subset X$ such that $\mu(X\setminus Z)<\varepsilon$ while $Z\subset\{g=h\}$.
By Lemma \ref{ue} we have
$$U_Z=\text{int}(\Pi^{-1}(Z)),\quad\tilde\mu(U_Z)=\mu(Z)>1-\varepsilon.$$
Since $Z\subset\{g=h\}$ then $\chi_Z\cdot(g-h)=0$. Consequently $\chi_{U_Z}\cdot(\hat g-\hat h)=\widehat{\chi_Z}\cdot(\hat g-\hat h)=0$
which implies
$$\{\hat g\ne\hat h\}\cap U_Z=\emptyset.$$

Put $Z_1:=Z$ and $\varepsilon=1/2$. Repeating the previous construction we find
a sequence $\{g_n\}\subset\cx$ and a sequence $\{Z_n\}$ of closed
subsets of $X$ such that $Z_n\subset\{g_n=h\}$ and $\mu(X\setminus Z_n)<1/2^n$.
Then $$\tilde{\mu}(U_{Z_n})=\mu(Z_n)>1-1/2^n,\quad \{\hat g_n\ne\hat h\}\cap U_{Z_n}=\emptyset.$$
The last equality implies that $\hat h$ is constant on each $\Pi^{-1}(\{x\})\cap U_{Z_n}$ for all $x\in X$ and $n\in\mathbb N$.
We define a sequence of open sets as follows:
$$U_1:=U_{Z_1},\quad U_n:=U_{Z_n}\setminus\Pi^{-1}(Z_1\cup...\cup Z_{n-1}).$$
By the above definition and Lemma \ref{open}, for $k\in\mathbb N$ we have
$\Pi^{-1}(Z_k)\supset U_{Z_k}\supset U_k$, hence $Z_k\supset \Pi(U_{Z_k})\supset \Pi(U_k)$,
and consequently
\begin{equation}\label{disj}
\Pi(U_n)\cap\Pi(U_m)=\emptyset\quad \for\ n\ne m
\end{equation}
since $\Pi(U_n)\cap Z_k=\emptyset$ for $k<n$.
By Lemma \ref{ue} we have $\tilde\mu(\Pi^{-1}(Z_n)\setminus U_{Z_n})=0$ and hence
\begin{multline}\label{mun}
\tilde\mu(U_n)=\tilde\mu(U_{Z_n}\setminus\Pi^{-1}(Z_1\cup...\cup Z_{n-1}))
=\tilde\mu(\Pi^{-1}(Z_n)\setminus\Pi^{-1}(Z_1\cup...\cup Z_{n-1}))\\
=\tilde\mu(\Pi^{-1}(Z_n\setminus(Z_1\cup...\cup Z_{n-1}))=\mu(Z_n\setminus(Z_1\cup...\cup Z_{n-1})).
\end{multline}
Put now $Z'_1:=Z_1$ and $Z'_n:=Z_n\setminus(Z_1\cup...\cup Z_{n-1})$ for $n>1$.
All the sets $\{Z_n'\}$ are pairwise disjoint and a direct calculation gives the equality
$Z'_n\cup Z'_{n-1}\supset Z_n\setminus(Z_1\cup...\cup Z_{n-2})$
which by induction leads to the assertion
$Z'_1\cup...\cup Z'_{n}\supset Z_n$. Hence, by \pref{mun} and pairwise disjointness of $\{U_n\}$ and $\{Z'_n\}$, we get
\begin{gather*}
\tilde\mu(U_1\cup...\cup U_n)=\tilde\mu(U_1)+...+\tilde\mu(U_n)=\mu(Z'_1)+...+\mu(Z'_n)\\
=\mu(Z'_1\cup...\cup Z'_n)\geqslant\mu(Z_n)>1-1/2^n.
\end{gather*}
Put $U:=\sum_{n=1}^{\infty}U_n$. Hence $U$ is open,
$\tilde\mu(U)=1=\tilde\mu(Y)$, and consequently, by Lemma \ref{gam1}, $U$ is dense in $Y$.
The function $\hat h$ is constant on each $\Pi^{-1}(\{x\})\cap U_n$ for all $x\in X$ and $n\in\mathbb N$ and sets $\Pi(U_n)$,
$n\in\mathbb N$ are pairwise disjoint by \pref{disj}. It means that each fiber $\Pi^{-1}(\{x\})$ intersects at most one of the sets $U_n$.
Hence $\hat h$ is constant on each $\Pi^{-1}(\{x\})\cap U$ for all $x\in X$.

\end{proof}

\begin{rem}
If the closed support of $\mu$ is not equal to $X$ then $\lmu$ is isometrically isomorphic to the algebra
$\{f_|{_{{\text{supp}(\mu)}}}:f\in\lmu\}$. In such a case  $\Pi^{-1}(\{x\})=\emptyset$ for all $x$ outside of
the closed support of $\mu$. Assuming that each function is constant on empty set we conclude that the result of Theorem
holds true also when the closed support of $\mu$ is a proper subset of $X$.
\end{rem}

\bibliographystyle{amsplain}

\ifx\undefined\bysame
\newcommand{\bysame}{\leavevmode\hbox to3em{\hrulefill}\,}
\fi

\end{document}